\documentclass[a4paper,12pt]{article}

\usepackage{orcidlink}
\usepackage{graphicx}%
\usepackage{multirow}%
\usepackage{amsmath,amssymb,amsfonts}%
\usepackage{amsthm}%
\usepackage{mathrsfs}%
\usepackage{hyperref}
\usepackage[title]{appendix}%
\usepackage{xcolor}%
\usepackage{textcomp}%
\usepackage{manyfoot}%
\usepackage{booktabs}%
\usepackage{textgreek}
\usepackage[ruled,vlined,linesnumbered]{algorithm2e}
\usepackage{complexity}
\usepackage{float}
\usepackage{tikz}
\usepackage{authblk}
\usepackage{fancybox}

\newcommand{\hull}[1]{\langle #1 \rangle}

\theoremstyle{thmstyleone}%
\newtheorem{theorem}{Theorem}
\newtheorem{lemma}[theorem]{Lemma}
\newtheorem{corollary}[theorem]{Corollary}
\theoremstyle{thmstyletwo}%

\theoremstyle{thmstylethree}%

\title{Partitions and covers in $\Delta$ convexity}\date{}

\author[1]{{Bijo S. Anand}\thanks{ORCID: 0000-0002-7221-904X}}

\author[2]{Manoj Changat\thanks{ORCID: 0000-0001-7257-6031}}

\author[3]{Mitre C. Dourado\thanks{ORCID: 0000-0001-9485-1073}}

\author[4]{\\ Prasanth G. Narasimha-Shenoi\thanks{ORCID: 0000-0002-5850-5410}}

\author[5]{Sabeer S. Ramla\thanks{ORCID: 0009-0001-8096-5499}}

\affil[1]{\footnotesize Department of Mathematics, Sree Narayana College, Punalur, Kollam, Kerala, India- 691035, bijos$\_$anand@yahoo.com}

\affil[2]{Department of Futures Studies, University of Kerala, Thiruvananthapuram, Kerala, India-695581
mchangat@keralauniversity.ac.in}

\affil[3]{Instituto de Computação, Universidade Federal do Rio de Janeiro, Rio de Janeiro, Brazil, mitre@ic.ufrj.br}

\affil[4]{Department of Mathematics, Government College Chittur, Palakkad, Kerala, Inida-678104, prasanthgns@gmail.com}
\affil[4]{Department of Collegiate Education, Government of Kerala, Thiruvananthapuram, Kerala, India-695033}

\affil[5]{Department of Applied Science and Engineering, Thangal Kunju Musaliar Institute of Technology, Musaliar Hills, Karuvelil, Kollam, Kerala, India-691505 sabeersainr@gmail.com}

\begin{document}

\maketitle

\begin{abstract}
Given a graph $G$ and a set $S \subseteq V(G)$, we say that $S$ is $\Delta$-convex if
there is no vertex $w \not\in S$ such that $w$ forms a triangle with two vertices of $S$.
Given an integer $p$, we are interested in the problems of finding a family of size $p$ of $\Delta$-convex sets that is a partition or a cover of $V(G)$.
We prove that both problems are \NP-complete for any fixed $p \ge 4$.
Next, we present a characterization of the graphs admitting a $\Delta$-convex $2$-partition.
Such characterization leads to a polynomial-time algorithm for finding a $\Delta$-convex $2$-partition if any.
Furthermore, we present some results relating these concepts with well-known ones like chromatic number and hull number and present closed formulas determining the minimum value $p$ for which the classical graph products admit a partition or a cover of $V(G)$ of size $p$.

\bigskip

\noindent\textbf{Keywords:} $\Delta$ convexity, $\Delta$-convex cover, $\Delta$-convex partition, graph products. \\ {\bf AMS Sub. Class.} 05C38,  05C76, 05C9, 52A01.
\end{abstract}

\section{Introduction}\label{introduction}

We consider finite, simple and undirected graphs. Given a graph $G$ and a set $S \subseteq V(G)$, the {\em $\Delta$-interval of $S$}, denoted by $[S]$, is the set formed by the vertices of $S$ and every $w \in V(G) \setminus S$ such that $w$ forms a triangle with two vertices of $S$. Sometimes, it will be useful to write $[u,v]$ standing for $[\{u,v\}]$. If $[S] = S$, then $S$ is a {\em $\Delta$-convex set} of $G$. The {\em $\Delta$-convex hull} of $S$, denoted by $\hull{S}$, is the minimum $\Delta$-convex set containing $S$. 
Let ${\cal F}$ be a family of size $p$ of $\Delta$-convex sets of $G$ whose union is $V(G)$.
If every member of ${\cal F}$ has an element that does not belong to any other member of ${\cal F}$, then ${\cal F}$ is a {\em $\Delta$-convex $p$-cover} of $G$.
If ${\cal F}$ is a cover and its members are disjoint, then ${\cal F}$ is a {\em $\Delta$-convex $p$-partition} of $G$.
In this work, we are interested in the following decision problems.

\bigskip\noindent\ovalbox{
\begin{minipage}{\dimexpr\textwidth-5\fboxsep-2\fboxrule\relax}
	\textbf{\sc $\Delta$-Convex $p$-cover} \\
	\begin{tabular}{lp{14cm}}
		Input: & A graph $G$. \\
		Question: & For a fixed integer $p$, does $G$ admit a $\Delta$-convex $p$-cover?
	\end{tabular}
\end{minipage}
}

\bigskip\noindent\ovalbox{
\begin{minipage}{\dimexpr\textwidth-5\fboxsep-2\fboxrule\relax}
	\textbf{\sc $\Delta$-Convex $p$-partition} \\
	\begin{tabular}{lp{14cm}}
		Input: & A graph $G$. \\
		Question: & For a fixed integer $p$, does $G$ admit a $\Delta$-convex $p$-partition?
	\end{tabular}
\end{minipage}
}
\bigskip

Unlike an independent set, whose any proper subset is also an independent set, a proper subset of a $\Delta$-convex set is not necessarily a $\Delta$-convex set. Such hereditary property of independent sets leads to the fact that if there is a partition of $V(G)$ into $p$ independent sets, then there is also a partition of $V(G)$ into $p+1$ independent sets as long as $p < |V(G)|$. In fact, this property does not hold in $\Delta$ convexity as can be seen in the graph of Figure~\ref{fig:nonmonotonicity}, which admits, for instance, a $\Delta$-convex $3$-partition but does not admit a $\Delta$-convex $4$-partition.

\begin{figure}[H] \label{fig:nonmonotonicity}
	\centering
	
\begin{tikzpicture}[scale=1]
	
	\pgfsetlinewidth{1pt}
	
	\tikzset{vertex/.style={circle,  draw, minimum size=5pt, inner sep=0pt}}
	\tikzset{texto/.style={circle,  minimum size=5pt, inner sep=0pt}}
	\tikzset{set/.style={circle,  draw, minimum size=30pt, inner sep=5pt}}
	
	\def\h{0}
	\def\v{0}
	
	\node [vertex] (x1) at (\h, \v  ) {};
	\node [vertex] (x2) at (\h-1, \v -1 ) {};
	\node [vertex] (x3) at (\h-1, \v+1  ) {};
	\node [vertex] (x4) at (\h+1, \v -1 ) {};
	\node [vertex] (x5) at (\h+1, \v+1  ) {};
	\node [vertex] (x6) at (\h+3, \v -1 ) {};
	\node [vertex] (x7) at (\h+3, \v+1  ) {};
	
	\draw (x1) to (x2);
	\draw (x1) to (x3);
	\draw (x1) to (x4);
	\draw (x1) to (x5);
	
	\draw (x3) to (x2);
	\draw (x4) to (x5);
	
	\draw (x6) to (x7);
	\draw (x6) to (x5);
	\draw (x6) to (x4);
	\draw (x6) to (x1);
	\draw (x7) to (x5);
	\draw (x7) to (x4);
	\draw (x7) to (x1);
	
\end{tikzpicture}

\caption{\footnotesize Example of a graph that
has a $\Delta$-convex $p$-partition for $p \in \{3,5,6,7\}$, but
has not a $\Delta$-convex $p$-partition for $p \in \{2,4\}$; and
has a $\Delta$-convex $p$-cover for $p \in \{2,3,5,6,7\}$, but
has not a $\Delta$-convex $p$-cover for $p = 4$.}

\end{figure}

However, note that any graph $G$ of order $n \ge 1$ admits a $\Delta$-convex $n$-partition, namely, the family containing a set for each vertex of $V(G)$.
Thus, for a graph $G$ of order $n \ge 2$,
we define the {\em $\Delta$-cover number} of $G$ as the minimum $p \ge 2$ such that $G$ admits a $\Delta$-convex $p$-cover, denoted by $\Phi_c(G)$; 
and we define the {\em $\Delta$-partition number} of $G$ as the minimum $p \ge 2$ such that $G$ admits a $\Delta$-convex $p$-partition, denoted by $\Theta_c(G)$.
Since every partition is a cover, we have for any graph $G$ that $\Phi_c(G) \le \Theta_c(G) \le |V(G)|$.

From now on, we consider only connected graphs of order at least $3$.
Given a graph $G$ and a set $S \subseteq V(G)$, the subgraph of $G$ induced by $S$ is denoted by $G[S]$. For $u \in V(G)$, the \emph{open neighborhood} of $u$ is denoted by $N(u)$.
We say that $u \in V(G)$ is a {\em $\Delta$-extreme vertex} if $V(G) \setminus \{u\}$ is $\Delta$-convex.
It is clear that if a graph $G$ has a $\Delta$-extreme vertex, then its $\Delta$-partition number is $2$.
Note that $u$ is $\Delta$-extreme if and only if $u$ does not belong to any triangle of $G$.

\subsection{Motivation and related works}

In the $P_3$ convexity a set $S$ is convex if there is no vertex outside $S$ having two neighbors in $S$.
Convex partitions of powers of cycles, split and chordal graphs in the $P_3$ convexity have been studied in~\cite{CDDRS2010}.
In the context of the well-known geodesic convexity, for a general graph $G$, it is known that the {\sc Convex $p$-cover} problem is \NP-complete for any fixed $p \geq 3$~\cite{artigas2010convex} and for $p = 2$~\cite{buzatu2015convex}; and that the {\sc Convex $p$-Partition} problem is \NP-complete for arbitrary graphs  for any fixed $p \geq 2$~\cite{artigas2011partitioning}.
On the other hand, it is solvable in polynomial time for bipartite graphs~\cite{grippo2016convex} and for cographs~\cite{artigas2011partitioning}.
The $\Delta$-cover number and the $\Delta$-partition number have been further investigated for disconnected graphs in \cite{artigas2010convex} and a characterization for chordal graphs is given in \cite{artigas2011partitioning}.
A rich source of convex cover and partition problems in the geodesic convexity is~\cite{buzatu2021covering} and of general abstract convexity is~\cite{van1993theory}.

It is worth noting that the problem of looking for a partition into two convex sets is very related with a well-known topic of abstract convexity theory, namely, the half-space separation problem. Recently, some papers are dedicated to this topic in graph convexities~\cite{elaroussi2024halfspaceseparationmonophonicconvexity, SEIFFARTH2023114105}. The input of such problem is a graph $G$ and two disjoint sets $A, B \subset V(G)$, and it asks whether there are disjoint convex sets containing $A$ and $B$ whose union is $V(G)$.

The $\Delta$-interval function and the associated $\Delta$-convexity were introduced by Mulder in~\cite{muldd}. In \cite{muld01}, the $\Delta$-convex sets, where they were called $\Delta$-closed sets, played an essential role in the characterization of quasi-median graphs.
Recall that the {\em $\Delta$-hull number} $h(G)$ of a graph $G$ is the cardinality of a minimum set  such that its $\Delta$-convex hull is $V(G)$. 
In \cite{bsaamc}, the authors studied the $\Delta$-hull number of a graph showing that computing this parameter is \NP-complete for general graphs, presenting polynomial-time algorithms for chordal graphs, dually chordal graphs and cographs and giving tight lower and upper bounds for this parameter. In \cite{anand2021delta}, the authors showed that the problem of computing the $\Delta$-convexity number is \W[1]-hard and \NP-hard to approximate within a factor $O(n^{1-\varepsilon})$ for any constant $\varepsilon > 0$ even for graphs with diameter $2$ and that the problem of computing the $\Delta$-interval number is \NP-complete for general graphs. For the positive side, they presented characterizations that lead to polynomial-time algorithms for computing the $\Delta$-convexity number of chordal graphs and for computing the $\Delta$-interval number of block graphs.

\subsection{Organization of the text}

This paper is organized as follows.
Section~\ref{np-complete} contains complexity results.
We begin by showing that the {\sc $\Delta$-convex $p$-partition} and {\sc $\Delta$-convex $p$-cover} problems are \NP-complete for any fixed $p \ge 4$. Next, we present a characterization of the graphs admitting a convex 2-partition. Such characterization leads to a polynomial-time algorithm for finding a convex 2-partition if any.
Section~\ref{sec:structural} contains structural results.
It has a subsection dedicated to these problems on graph classes related with chordal graph, relating partitions and covers with the chromatic number and the hull number.
It also has a subsection which deals the partitions and the cover in the three standard graph products, namely, the Cartesian, strong and lexicographic products.
We discuss some open problems in Section~\ref{sec:conclusion}.

\section{Complexity aspects} \label{np-complete}

In this section, we show that {\sc Convex $p$-partition} and {\sc Convex $p$-cover} are \NP-complete for any fixed $p \ge 4$ and present a polynomial-time algorithm for solving the {\sc $\Delta$-Convex $2$-partition} problem.

\begin{theorem} \label{the:NPCpartition}
	{\sc $\Delta$-Convex $p$-partition} is \NP-complete for any fixed $p \ge 4$.
\end{theorem}

\begin{proof}
Since to check whether a given set is convex in $\Delta$-convexity can be done in polynomial time, one can test in polynomial time whether each set of a given partition is indeed a convex set, which implies that {\sc $\Delta$-Convex $p$-partition} belongs to \NP.

For the hardness part, we present a reduction from {\sc $k$-Colorability}. It is known that {\sc $k$-Colorability} is \NP-complete for any fixed integer $k \ge 3$~\cite{GAREY1976237}. Let $G$ be an input for the {\sc $k$-Colorability} problem for a fixed $k \ge 3$. Clearly, we can assume that $G$ is a non-trivial connected graph. Let $G'$ be the graph obtained from $G$ by adding one new vertex $u$ and the edges to make $u$ a universal vertex of $G'$.

Now, observe that every pair of adjacent vertices of $G'$ is a hull set of $G'$.
Observe also that if every pair of adjacent vertices is a hull set of $G$, then the only proper convex sets of $G$ are the independent sets.
Therefore, the problem of partitioning $V(G')$ into convex sets is equivalent to the problem of partition $V(G)$ into independent sets. Since $G'$ has a universal vertex, any partition of $V(G')$ into convex sets has the singleton $u$, which implies that $G$ is $k$-colorable if and only if $G'$ has a convex $(k+1)$-partition.
\end{proof}

\begin{corollary}
	{\sc $\Delta$-Convex $p$-cover} is \NP-complete for any fixed $p \ge 4$.
\end{corollary}

\begin{proof}
Let $G'$ be the graph constructed in the proof of Theorem~\ref{the:NPCpartition} from the input graph $G$.
It suffices to show that $G'$ has a convex $p$-partition if and only if $G'$ has a convex $p$-cover.
By definition, every $p$-partition is a $p$-cover.
Therefore, it remains to prove that every $\Delta$-Convex $p$-cover ${\cal C}$ of $G'$ can be transformed into a $\Delta$-Convex $p$-partition ${\cal P}$ of $G'$. Partition ${\cal P}$ can be obtained by applying the following process iteratively. For every vertex $v$ that belongs to more than one part of ${\cal C}$, remove $v$ from all parts containing $v$ except one.
\end{proof}

Now, we focus on the problem of finding a $\Delta$-convex $p$-partition for $p = 2$.
Before presenting the characterization of the graphs that admits a $\Delta$-convex 2-partition and the polynomial-time algorithm that follows from it, we begin by characterizing the $\Delta$-convex 2-partitions.

\begin{lemma} \label{lem:atom}
Let $G$ be a graph and let $\{S_1,S_2\}$ be a $2$-partition of $V(G)$.
Then, $\{S_1,S_2\}$ is a $\Delta$-convex $2$-partition of $G$ if and only if every triangle $T$ of $G$ satisfies $T \subseteq S_i$ and $T \cap S_j = \emptyset$ for some $i \in \{1,2\}$ and $j \in \{1,2\} \setminus \{i\}$.
\end{lemma}

\begin{proof}
First, consider that $S_1$ and $S_2$ are $\Delta$-convex.	
Suppose by contradiction that the statement is not satisfied by a triangle $T = \{v_1,v_2,v_3\}$.
Since $\{S_1,S_2\}$ is a partition of $V(G)$, every vertex belongs to exactly one of these two sets.	
Then, without loss of generality, we can write $T \cap S_1 = \{v_1,v_2\}$ and $T \cap S_2 = \{v_3\}$. But, since $S_1$ is a $\Delta$-convex set and $v_1v_2 \in E(G)$, we have that $v_3 \in S_1$, which is a contradiction.

Conversely, consider that $S_1$ or $S_2$ is not $\Delta$-convex. Without loss of generality, we can assume that $S_1$ is not $\Delta$-convex, which means that there is a triangle $T$ such that $|T \cap S_1| = 2$ and $|T \cap S_2| = 1$, completing the proof.
\end{proof}

We need of the following definition. Given a graph $G$, construct a graph $T(G)$ as follows.

\begin{itemize}
	\item for every triangle of $G$, add a vertex in $T(G)$;
	\item for every vertex of $G$ not belonging to any triangle in $G$, add a vertex in $T(G)$; and
	\item add an edge in $T(G)$ if the corresponding vertices are associated with triangles of $G$ and share at least one vertex.
\end{itemize}

We can now present the characterization of the graphs admitting a $\Delta$-convex $2$-partition.

\begin{theorem} \label{the:characC2P}
A graph $G$ admits a $\Delta$-convex $2$-partition if and only if $T(G)$ is disconnected.
\end{theorem}

\begin{proof}
First, consider that $G$ has a $\Delta$-convex $2$-partition $\{S_1,S_2\}$.
Suppose by contradiction that $G$ has no $\Delta$-extreme vertices and $T(G)$ is connected.
Since $S_1$ and $S_2$ are non-empty sets, let $v_1 \in S_1$ and $v_2 \in S_2$.
The assumption that $G$ has no $\Delta$-extreme vertices implies that $G$ has a triangle $T_i$ containing $v_i$ for $i \in \{1,2\}$.
Denote by $t_i$ the vertex of $T(G)$ associated with $T_i$ for $i \in \{1,2\}$.
Since $T(G)$ is connected, there is a $(t_1,t_2)$-path $P$ in $T(G)$.
Because of Lemma~\ref{lem:atom}, we know that the triangles $T_1$ and $T_2$ are disjoint. 
Moreover, $P$ contains vertices $t'_1$ and $t'_2$ such that $t'_1t'_2$ is an edge of $T(G)$
and the corresponding triangles $T'_1$ and $T'_2$
satisfy $T'_i \subseteq S_i$ and $T'_j \cap S_i = \emptyset´$ for $i \in \{1,2\}$, which means that $T'_i \cap T'_2 = \emptyset$. However, the construction implies that $T'_1$ and $T'_2$ share at least one vertex, which is a contradiction.

Conversely, consider that $G$ has $\Delta$-extreme vertices.
As discussed above, if $v$ is a $\Delta$-extreme vertex, then $\{\{v\}, V(G) \setminus \{v\}\}$ is a $\Delta$-convex $2$-partition.

Now, consider that $G$ has no $\Delta$-extreme vertices but $T(G)$ is disconnected. 
Let $C_1$ be a connected component of $T(G)$. Define $S_1$ as the set formed by the vertices of the triangles of $G$ corresponding to the vertices of $C_1$.
We claim that $\{S_1, V(G) \setminus S_1 \}$ is a $\Delta$-convex 2-partition of $G$. 
Since $G$ has no $\Delta$-extreme vertices, every vertex $v$ of $G$ belongs to at least one triangle.
By the construction, the vertices of $T(G)$ associated with the triangles of $v$ belong to the same connected component. Therefore, $\{S_1, V(G) \setminus S_1 \}$ is a partition of $V(G)$.
It remains to show that these two sets are $\Delta$-convex.
Suppose by contradiction that $S_1$ is not $\Delta$-convex.
Therefore, there are vertices $u,v \in S_1$ and $w \not\in S_1$ such that $T = \{u,v,w\}$ is a triangle of $G$. Denote by $t$ the vertex of $T(G)$ associated with $\{u,v,w\}$.
Since $u,v \in S_1$, there is a vertex $t' \in C_1$ corresponding to a triangle $T' \ne T$ in $G$ containing $u$. By the construction of $T(G)$, $t'$ is adjacent to $t$, which implies that $t'$ belongs to $C_1$ and that $w \in S_1$, which is a contradiction.

The proof that $V(G) \setminus S_1$ is $\Delta$-convex is analogous.
\end{proof}

Next, we present the polynomial-time algorithm for finding a $\Delta$-convex 2-partition of a general graph if any. This algorithm follows from the above characterization.

\begin{algorithm}[h] \label{alg:convex2partition}
	
	\caption{{\sc Convex2Partition}}
	
	\KwIn{A graph $G$ of order $n \ge 2$.}
	\KwOut{A convex $2$-partition of $V(G)$ if one there exists; empty set otherwise}

	\For{$v \in V(G)$}{ \label{lin:enumeratingK3}
		${\cal T}_v = $ list of triangles containing $v$
	}
	
	\If{$G$ has a $\Delta$-extreme vertex $v$}{ \label{lin:extreme}
		$S_1 = \{v\}$

		$S_2 = V(G) \setminus S_1$

		\Return $\{S_1,S_2\}$ \label{lin:notriangles}
	}
	
	let $S_1$ be a triangle of $G$ \label{lin:T}
	
	\While{there is a triangle $T$ of $G$ such that $1 \le |T \cap S_1| \le 2$}{ \label{lin:while}

		$S_1 = S_1 \cup T$ \label{lin:S1}
	}	
	\If{$S_1 \ne V(G)$}{ \label{lin:afterwhile}
		$S_2 = V(G) \setminus S_1$
			
		\Return $\{S_1,S_2\}$ \label{lin:2partition}
	}
	\Return $\emptyset$ \label{lin:no2partition}

\end{algorithm}

\begin{theorem}
Algorithm~$\ref{alg:convex2partition}$ is correct.
\end{theorem}

\begin{proof}
According to Theorem~\ref{the:characC2P}, we need to check whether $G$ has at least one $\Delta$-extreme vertex or $T(G)$ is disconnected. The existence of $\Delta$-extreme vertices is done in line~\ref{lin:extreme}. The test of the connectedness of $T(G)$ is done implicitly in lines~\ref{lin:enumeratingK3} to~\ref{lin:no2partition} without its construction.
This test can be divided into the following three steps:
(1) to choose any vertex $t$ of $T(G)$, next (2) finding all vertices of the connected component $C$ containing $t$, and (3) test whether $V(C)$ is equal to $V(T)$.

Line~\ref{lin:enumeratingK3} is reached only if $G$ has no $\Delta$-extreme vertices, which means that every vertex of $G$ belongs to at least one triangle and that $T(G)$ has at least one vertex.
In line~\ref{lin:T}, one triangle of $G$ is chosen, which is equivalent to choose any vertex $t$ of $T(G)$, which implies that (1) is done.
In the while loop beginning in line~\ref{lin:while}, all vertices of the triangles sharing vertices to the triangles contained in $S_1$ are added to $S_1$ iteratively. Therefore, when this while loop finishes, $S_1$ contains all vertices of all triangles that are associated with the vertices of the connected component $C_1$, since two vertices are adjacent in $T(G)$ if and only if the corresponding triangles in $G$ share vertices, then (2) is completed.

When line~\ref{lin:afterwhile} is reached, the sets $S_1$ and $V(G) \setminus S_1$ satisfy the property that for every triangle $T$ of $G$, either $T \subseteq S_1$ or $T \subseteq V(G) \setminus S_1$. Then, according to Lemma~\ref{lem:atom}, $\{S_1, V(G) \setminus S_1\}$ is a $\Delta$-convex $2$-partition of $G$ as long as $V(G) \setminus S_1$ be non-empty. This test is done in lines~\ref{lin:afterwhile} to~\ref{lin:no2partition}, achieving $(3)$.
\end{proof}

\begin{theorem}
The time and memory complexities of Algorithm~$\ref{alg:convex2partition}$ is asymptotically equivalent to the task of enumerating the triangles of the input graph.
\end{theorem}

\begin{proof}
In line~\ref{lin:enumeratingK3}, an algorithm for enumerating all triangles of the input graph is used. In fact, with a variation, which consists that at the moment that a triangle is found by such algorithm, it is added in the list of triangles of each of its vertices. Each such list is initially empty. This tast is asymptotically equivalent to enumerating the triangles of the input graph.

Line~\ref{lin:extreme} is done in linear time since $G$ has a $\Delta$-extreme vertex $v$ if and only if at the list of triangles of $v$ is empty.

If $G$ contains triangles, the first found in line~\ref{lin:enumeratingK3} is stored to be used line~\ref{lin:T}, then this line is clearly done in linear time.

Now, consider the while loop beginning in line~\ref{lin:while}. The algorithm considers the vertices added to $S_1$ as queue, so that each element of each list of triangles of these vertices is considered only once. Since each triangle is formed by three vertices, each triangle is considered at most three times in this loop. To avoid that a vertex be added to $S_1$ twice, the algorithm can use a boolean $n$-vector to represent the vertices already added to $S_1$. Therefore, line~\ref{lin:S1} can be done in constant time and the whole while loop is asymptotically equivalent to the task of enumerating the triangles of the input graph.

Finally, it is clear that lines~\ref{lin:afterwhile} to~\ref{lin:no2partition} can be done in linear time, completing the proof.
\end{proof}

\section{Structural results}\label{sec:structural}

We begin this section by presenting some results relating the parameters $\Phi_c(G)$ and $\Theta_c(G)$ with well-known results like chromatic number and hull number. We conclude this section presenting closed formulas for determining these parameters on graph products.

A connected graph $G$ with no cut-vertices is said to be {\em $2$-connected}.  
The \emph{chromatic number} of a graph $G$ can be defined as a minimum number of colours needed to colour $G$ in such a way that no two adjacent vertices have the same colour, denoted as $\chi(G)$.  A graph $G$ is \textit{chordal}
if $G$ have no induced cycles $C_n$ with length $n$ more than three.  A graph is \textit{dually	chordal} if it has a maximum neighborhood ordering. A \textit{block graph} is a graph in which each block is complete.  Note that every block graph is a chordal graph.

We begin this section by considering a block graph $G$ of order $n \ge 2$. Note that if $G$ is a complete graph, then $\Phi_{c}(G) = n$; and that if $G$ is not complete, then $\Phi_{c}(G) = 2$. The next result determines $\Theta_c(G)$ for block graphs.

\begin{theorem}
	If $G$ is a connected block graph of order $n \ge 2$, then $\Theta_c(G)$ is the order of the smallest block of $G$.
\end{theorem}

\begin{proof}
	Let $p$ be the order of a  smallest block $B$ of $G$. Write $V(B)=\left\{v_{1}, v_{2}, \ldots, v_{p}\right\}$. Now we divide $V(G)$ into $p$ sets ${\cal V}=\left\{V_{1}, V_{2}, \ldots, V_{p}\right\}$ as following. For $i \in \{1, \ldots, p\}$, define $V_i = V(G_{v_i,B})$ where $G_{v_i,B}$ is the connected component of $V(G)\setminus (V(B) \setminus  \{v_i\})$ containing $v_i$ whenever $v_i$ is cut vertex of $G$ and
	$V_i = \{v_i\}$ otherwise. For $i \in \{1, \ldots, p\}$, $V_{i}$ is convex because there is not triangle in $G$ with exactly two vertices in $V_i$. Therefore ${\cal V}=\left\{V_{1}, V_{2}, \ldots, V_{p}\right\}$ is a convex partition of $G$.
	
	For $p = 2$, it is clear that $\Theta_c(G) = 2$.
	From now on, consider that $p \ge 3$.  Assume for contradiction that $1 < m < p$. Let $B$ be a block of $G$. We have either $B \subseteq U_i$
	for some $i$ or $|B \cap U_i | \leq 1$ for each $1 \leq i \leq m$. Since $m < p$ by assumption, there must exist
	$u, v \in B$ such that $u, v \in U_i$ for some $i$. We obtain $B \subseteq U_i$ . In other words, each block of $B$ is
	contained in some $U_i$. Moreover, each $B$ is contained in the same $U_i$ as all its incident blocks. Since
	$G$ is connected, we deduce $V(G) = U_i$ by repeating the previous argument up to saturation. This
	contradicts $\{U_1 , \ldots , U_m\}$ being a convex partition of $V(G)$.
\end{proof}

Now, we consider chordal and dually chordal graphs. We need the following theorem from \cite{bsaamc}.

\begin{theorem} \label{the:2conDuaChordal} {\em \cite{bsaamc}}
	Let $G$ be a $2$-connected graph. If $G$ is chordal or dually chordal, then every pair of adjacent vertices form a $\Delta$-hull set of $G$.
\end{theorem}

For a graph $G$ in which every pair of adjacent vertices is a hull set, the only proper convex sets of $G$ are the independent sets. In that case the convex cover number and the convex partition number of $G$ are equal and coincide with the chromatic number $\chi(G)$. Using Theorem~\ref{the:2conDuaChordal}, we have the following equalities.

\begin{corollary}
	If $G$ is a chordal or dually chordal graph, then $\Phi_{c}(G) = \Theta_c(G) = \chi(G)$.
\end{corollary}

The next result presents a property of the graphs admitting a convex 2-cover.

\begin{lemma}\label{lem:p2}
	Let $G$ be a graph of order $n \geq 3$. If $\Phi_{c}(G) = 2$, then $h(G) \geq 3$.
\end{lemma}

\begin{proof}
	Suppose for contradiction that $h(G) = 2$ and that $G$ has non-comparable convex sets $V_1$ and $V_2$ such that $V_1 \cup V_2 = V(G)$. Let $u_1,u_2 \in V(G)$ with $\langle u_1,u_2 \rangle = V(G)$. Since $|V(G)| \ge 3$, there is a vertex $u_3$ forming a triangle with $u_1$ and $u_2$. Since $\{u_1,u_2\}$ is a hull set, we can assume without loss of generality that $u_1 \in V_1$ and $u_2,u_3 \in V_2$. Note that $u_1 \in [V_2]$, which contradicts the assumption that $V_2$ is a convex set.
\end{proof}

Since $\Phi_{c}(G) \le \Theta_c(G) = 2$, Lemma~\ref{lem:p2} implies that if $\Theta_c(G) = 2$ and $n \ge 3$, then $h(G) \geq 3$. In general, the converse of the Lemma~\ref{lem:p2} does not hold. See Figure~\ref{fig:lemma2}. \\ 

\begin{figure}[H]
	\centering
	
	\begin{tikzpicture}[scale=1.2]
		
		\pgfsetlinewidth{1pt}
		
		\tikzset{vertex/.style={circle, draw, minimum size=5pt, inner sep=0pt}}
		
		\def\h{2}
		\def\v{1}
		
		\node [vertex] (u1) at (\h, \v) {$u_1$};
		\node [vertex] (u2) at (\h, \v-1) {$u_2$} edge (u1);
		\node [vertex] (u3) at (\h-2, \v) {$u_3$} edge (u1);
		\node [vertex] (u4) at (\h-2, \v-1) {$u_4$} edge (u2) edge (u3);
		
		\draw (u1) -- (u4);
		\draw (u2) -- (u3);
		
		\def\h{2}
		\def\v{-3}
		
		\node [vertex] (v1) at (\h, \v) {$v_1$};
		\node [vertex] (v2) at (\h, \v-1) {$v_2$} edge (v1);
		\node [vertex] (v3) at (\h-2, \v) {$v_3$} edge (v1);
		\node [vertex] (v4) at (\h-2, \v-1) {$v_4$} edge (v2) edge (v3);
		
		\draw (v1) -- (v4);
		\draw (v2) -- (v3);
		
		\draw (u2) -- (v1);
		
		\def\h{4}
		\def\v{-2}
		
		\node [vertex] (w3) at (\h, \v+2) {$w_3$} edge (u2);
		\node [vertex] (w4) at (\h, \v-1) {$w_4$} edge (v1) edge (w3);
		
		\node [vertex] (x1) at (\h-8, \v+3) {$x_1$} edge (v3) edge (u3) edge (w3);
		\node [vertex] (x2) at (\h-8, \v+2) {$x_2$} edge (v3) edge (u4) edge[bend left=45] (w3);
		\draw (x2) -- (x1);
		\draw (w3) -- (v3);
		
		\draw (u3) -- (x2);
		\draw (u4) -- (x1);
		
		\node [vertex] (z1) at (\h-8, \v-1) {$z_1$} edge[bend right=45] (w4) edge (u4) edge (v3);
		\node [vertex] (z2) at (\h-8, \v-2) {$z_2$} edge (w4) edge (u4) edge (v4);
		\draw (z2) -- (z1);
		\draw (w4) -- (u4);
		
		\draw (v3) -- (z2);
		\draw (v4) -- (z1);
		
	\end{tikzpicture}
	
	\caption{Graph $G$ such that $h(G) \ge 3$ and $\Phi_c(G) \ge 3$.}
	\label{fig:lemma2}
\end{figure}


\subsection{Graph Products}\label{graph_product}

For the three graph products considered here on graphs $G$ and $H$, the vertex set is $V(G)\times V(H)$. Their edge sets are defined as follows. In the \emph{Cartesian product}, $G\Box H$, two vertices are adjacent if they are adjacent in one coordinate and equal in the other. 
Two vertices $(g_1, h_1)$ and $(g_2, h_2)$ are adjacent with respect to the \emph{strong product}, $G \boxtimes H$, if $(a)$ $g_1 = g_2$ and $h_1h_2 \in E(H),$ or $(b)$ $h_1 = h_2$  and  $g_1g_2\in E(G)$, or $(c)$ $g_1g_2\in E(G)$ and $h_1h_2\in E(H).$
Finally, two vertices $(g_1,h_1)$ and $(g_2,h_2)$ are adjacent in the \emph{lexicographic product}, $G \circ H$ if, either $g_1g_2\in E(G)$ or $g_1=g_2$ and $h_1h_2\in E(H)$.

Let $G$ and $H$ be graphs and let $v \in V(G)$. Recall that $v$ is a {\em $\Delta$-extreme vertex} if $V(G) \setminus \{v\}$ is $\Delta$-convex. Observe that if both $G$ and $H$ contain $\Delta$-extreme vertices, then $G \Box H$ also contains $\Delta$-extreme vertices, which implies that $\Phi_c(G \Box H)= \Theta_c(G \Box H)= 2 $.

\begin{theorem} \label{boxcover}
	Let $G$ and $H$ be two connected non-trivial graphs such that at least one of them has a $\Delta$-extreme vertex. Then $\Phi_{c}(G \Box H) = \Theta_{c}(G \Box H) =  2 $.
\end{theorem}

\begin{proof} Assume that $G$ has a $\Delta$-extreme vertex, say $x$. Then any vertex in $\{x\} \times V(H)$ will not form a triangle with the vertices of $ V(G) \setminus (\{x\} \times V(H))$. Then $ S_1= \{x\} \times V(H)$ and $S_2= V(G) \setminus (\{x\} \times V(H))$ will be two proper convex sets in $G \Box H $ and $ S_1 \cup  S_2 = V(G \Box H)$. Therefore $\Phi_{c}(G \Box H) =  2$.
\end{proof}

\begin{theorem}\label{covercart}
	Let $G$ and $H$ be two connected non-trivial graphs. Then $\Phi_{c}(G \Box H) = \min\{\Phi_c(G), \Phi_c(H)\}$ and $\Theta_{c}(G \Box H) = \min\{\Theta_c(G), \Theta_c(H)\}$.
\end{theorem}

\begin{proof} 
	Let $m$ and $n$ be the convex cover numbers of $G$ and $H$, respectively. We have to prove $\Phi_{c}(G \Box H) \leq \min\{m, n\}$. Since $\Phi_c(G)=m$, there exists $m$ proper convex sets $V_1,V_2,\ldots,V_m$ of $G$ with $V_1\cup V_2\cup \ldots \cup V_m=V(G)$. Since each $V_i$ is a proper convex sets in $G$ for $i \in \{1, \ldots, m\}$, no edge of $V_i$ forms a triangle with any vertex of $V_j$ for $j \in \{1, \ldots, m\} \setminus\{i\}$. By the definition of Cartesian product of graphs, no edge of $V_i\times V(H)$ forms a triangle with any vertices of $V_j\times V(H)$ for $j \in \{1, \ldots, m\} \setminus \{i\}$. So each $V_i\times V(H)$ is a proper convex set in $G\Box H$ and $\displaystyle \cup_{i=1}^m V_i\times V(H)=V(G\times H)$. Since the Cartesian product is commutative, it holds that $\Phi_{c}(G \Box H) \leq \min\{\Phi_c(G), \Phi_c(H)\}$.
	
	Assume for a moment that $\Phi_c(G)<\Phi_c(H)$. Let $\{V_1, V_2\ldots, V_s \}$ be a convex partition set of $G\Box H$ with $s<\Phi_c(G)$. Since $\{V_1, V_2\ldots, V_s \}$ is a convex partition set of $G\Box H$, $\{\pi_G(V_1), \pi_G(V_2)\ldots, \pi_G(V_s) \}$ is a convex partition set of $G$, which contradicts to the fact that $s<\Phi_c(G)$. Therefore $\Phi_c(G\Box H)=\Phi_c(G)$. Similarly we can prove that $\Phi_c(G\Box H)=\Phi_c(H)$, when $\Phi_c(H)<\Phi_c(G)$. Giving the similar proof we can conclude that $\Phi_c(G\Box H)=\Phi_c(G)=\Phi_c(H)$, when $\Phi_c(G)=\Phi_c(H)$. Therefore $\Phi_{c}(G \Box H) = \min\{\Phi_c(G), \Phi_c(H)\}$.
	
	By giving a similar argument, we can conclude that \\ $\Theta_{c}(G \Box H) = \min\{\Theta_c(G), \Theta_c(H)\}$.	
\end{proof}

\begin{corollary}
	Let $G$ and $H$ be two non-trivial connected graphs such that at least one of them contains cut vertices. Then $\Phi_{c}(G \Box H) = 2 $.
\end{corollary}

\begin{proof}
	Since the Cartesian product is commutative, we can assume that $G$ has a cut vertex $v$. Let $C_1$ and $C_2$ be connected components of the induced subgraph of $V(G) \setminus \{v\}$. Take $V_1 = V(C_1) \cup \{v\}$ and $V_2 = V(G) \setminus V(C_1)$. Then $V_1$ and $V_2$ are convex sets in $G$ and $V_1\cup V_2=V(G)$. No two vertices of $V_1$ forms a triangle with the vertices of $V_2$ and vice versa. By the definition of Cartesian product of graphs, no two vertices of $V_1\times V(H)$ form a triangle with the vertices of $V_2\times V(H)$. Similarly no two vertices of $V_2\times V(H)$ form a triangle with the vertices of $V_1\times V(H)$. As a result $V_1\times V(H)$ and $V_2\times V(H)$ are two proper convex sets in $G\Box H$ and $(V_1\times V(H))\cup (V_2\times V(H))=V(G\Box H)$. Therefore $\Phi_{c}(G \Box H) = 2 $. \end{proof}
\begin{corollary}
	\begin{enumerate}
		\item For $m,n\geq 2$, $\Phi_c(K_m\Box K_n)=\min\{m,n\}$.
		\item  For $m,n\geq 2$, $\Phi_c(T_m\Box T_n)=2$, where $T_m$ and $T_n$ are trees of order $m$ and $n$, respectively. 
		\item For $m,n\geq 2$, $\Phi_c(C_m\Box C_n)=2$, where $C_m$ and $C_n$ are cycles of order $m$ and $n$, respectively. 
		
	\end{enumerate}
\end{corollary}

\begin{theorem}\label{GH} {\em \cite{anand2021delta}}
	Let $G$ and $H$ be two connected non-trivial graphs. Then every two adjacent vertices form a $\Delta$-hull set of $G \ast H$ for $\ast\in\{\boxtimes, \circ\}$.
\end{theorem}

The following theorem holds from Theorem~\ref{GH}.

\begin{theorem}
	Let $G$ and $H$ be two connected non-trivial graphs. Then $\Phi_{c}(G \ast H) = \Theta_{c}(G \ast H) = \chi(G\ast H)$ for $\ast\in\{\boxtimes, \circ\}$.
\end{theorem}

\section{Open problems} \label{sec:conclusion}

We conclude by calling attention to the remaining open cases. The {\sc $\Delta$-convex $p$-partition} is still open for $p = 3$ and the The {\sc $\Delta$-convex $p$-cover} is open for $p \in \{2,3\}$.

\section*{Acknowledgements}

Mitre C. Dourado is partially supported by
Conselho Nacional de Desenvolvimento Científico e Tecnológico (CNPq), Brazil, grant numbers 403601/2023-1 and 305141/2024-4.

\section*{Declaration}
{\bf Conflict of interest:} The authors declare that they have no known competing financial interests or personal relationships that could have appeared to influence the work reported in this paper.
\bibliographystyle{amsplain}
\bibliography{partition}

\end{document}